\documentclass[12pt]{amsart}
\usepackage{amssymb,latexsym,amsmath,amscd,amsthm,graphicx, color}
\usepackage[all]{xy}
\usepackage{pgf,tikz}
\usepackage{mathrsfs}
\usetikzlibrary{arrows}
\definecolor{uuuuuu}{rgb}{0.26666666666666666,0.26666666666666666,0.26666666666666666}
\definecolor{xdxdff}{rgb}{0.49019607843137253,0.49019607843137253,1.}
\definecolor{ffqqqq}{rgb}{1.,0.,0.}

\newtheorem{theorem}{Theorem}[section]
\newtheorem{lemma}[subsection]{Lemma}

\theoremstyle{definition}
\newtheorem{remark}[subsection]{Remark}
\newtheorem{definition}[subsection]{Definition}
\newtheorem{example}[subsection]{Example}

\theoremstyle{remark}

\numberwithin{equation}{section}

\begin{document}

\title{Effect of the Riemann-Liouville fractional integral on unbounded variation points}


%
\author{S. Verma}
\address{Department of Mathematics, IIT Delhi, New Delhi, India 110016 }
\email{saurabh331146@gmail.com}
\author{Y. S. Liang}
\address{Institute of Science, Nanjing University of Science and Technology, Nanjing 210094, China}
\email{80884903@qq.com}







\begin{abstract}
 
This paper targets to study the effect of the Riemann-Liouville fractional integral operator on unbounded variation points of a continuous function. In particular, we show that the fractional integral preserves the bounded variation points of a function. We also prove that the fractional integral operator is a bounded linear operator on the class of bounded variation and continuous functions.
\end{abstract}

\maketitle


.

\section{INTRODUCTION}
Let us start our discussion as follows:
 for a continuous function $f:[a,b] \rightarrow \mathbb{R},$ the function $F(x):= \int_{a}^{x} f(t)dt$ is of bounded variation, and hence $\dim_{B}(Gr(F)) = \dim_{H}(Gr(F))=1,$ where $Gr(F)$ denotes the graph of $F$, 
 $\dim_B(.)$ and $\dim_H(.)$ denote respectively the box dimension and Hausdorff dimension of a set. We cite the books \cite{B} and \cite{PM1} for definitions of dimensions and their properties. Here and throughout the paper, we write $\mathcal{C}([a,b])$ to represent the collection of all continuous functions defined on $[a,b].$
	 Let $f \in \mathcal{C}([a,b])$ and $0<\alpha<1$. Let $\mathfrak{I}_{a}^{\alpha}f(a)=0$, and for $x\in (a,b]$, we define
	\begin{equation}
	\mathfrak{I}_{a}^{\alpha}f(x)=\frac{1}{\Gamma(\alpha)}\int_{a}^{x}(x-t)^{\alpha-1}f(t)dt.
	\end{equation}
	This is the Riemann-Liouville fractional integral of $f$ of order $\alpha$. A natural question arises here about the dimension of fractional integral of a function. There have been several attempts to answer the question partially, see \cite{Liang1,WY,VV1,WD,Z}.
	A major progress towards the problem is achieved by Liang \cite{Liang11} that for any continuous functions whose box dimension is one on $[0, 1],$ Riemann-Liouville fractional integral of these functions of any positive order has been proved to still be $1$-dimensional continuous functions on $[0, 1].$ However, to the best of our knowledge, the existing literature could not serve one complete answer to it. In this paper, we also make a modest contribution aiming to some new advances in the field. 
	\par
	
	If $\alpha=1$, we have $\mathfrak{I}_{a}^{1}f(x)=\int_{a}^{x}f(t)\mathrm{d}t$, and hence  $\dim_{B}(Gr(\mathfrak{I}_{a}^{\alpha}f))=\dim_{H}(Gr(\mathfrak{I}_{a}^{\alpha}f)) =1$.
	For $0<\alpha<1$, Liang \cite{Liang1} proved that if $f \in \mathcal{C}([a,b])$ and is of bounded variation, then we have $$\dim_B (Gr(\mathfrak{I}_{a}^{\alpha}f))=\dim_H (Gr(\mathfrak{I}_{a}^{\alpha}f))=1.$$
Following this result of Liang, authors of \cite{VV1} extend the result to the Katugaompola fractional integral, a generalization of the Riemann-Liouville fractional integral and Hadamard fractional integral. Continuing in the same vein, they further obtain a similar result in the bivariate setting with respect to an appropriate definition of bounded variation, see, for instance, \cite{VV2}. In particular, they establish:
\begin{theorem}
Let $f :[a,b] \times [c,d] \rightarrow \mathbb{R}$ be a continuous and of bounded variation in the sense of Arzel\'a. Then, for $\alpha,\beta>0,$ we have  $$\dim_B(Gr(\mathfrak{I}_{(a,c)}^{(\alpha,\beta)}f))=\dim_H(Gr(\mathfrak{I}_{(a,c)}^{(\alpha,\beta)}f))=2,$$
where $\mathfrak{I}_{(a,c)}^{(\alpha,\beta)}f$ denotes the (mixed) Riemann-Liouville fractional integral of $f$, and is defined as $$ \mathfrak{I}_{(a,c)}^{(\alpha, \beta)}f(x,y)=\frac{1}{\Gamma (\alpha)  \Gamma (\beta)} \int_a ^x \int_c ^y (x-s)^{\alpha-1} (y-t)^{\beta-1}f(s,t)~dsdt .$$
\end{theorem}
Among the various problems associated to the subject of Fractional calculus, an important research problem that received considerable interest is to calculate the box dimension and the Hausdorff dimension of the graph of the fractional integral of a function. The reader is encouraged to see \cite{Liang2,LX,VV1,Z} for researches related to the box dimension of the graph of a continuous function which is not of bounded variation. More recently, the box dimension of the graph of the Hadamard fractional integral of a continuous function of bounded variation and not of bounded variation is investigated in \cite{WD}. On the other hand, the relationship between unbounded variation points of a continuous function and its dimension has also been a topic of research in the intersection of fractal geometry and the notion of bounded variation, see, for instance, \cite{Liang2,Liang4,Z}. In the present paper, we shed some lights on this area of research. 
 \par
 The structure of the paper is as follows. 
In the main section, we prove that the fractional operator preserves bounded variation points of a function. Further, we show that the fractional operator is a bounded linear operator when it is restricted to the class of continuous and bounded variation functions.

\section{Main Results}
First, we equip this section with some definitions related to our concern.
\begin{definition}
Let $ f:[a,b] \rightarrow \mathbb{R}$ be a function. For each partition $ P: a=t_0<t_1<t_2 < \dots <t_n =b$ of the interval $[a,b], $ we define $$V(f,[a,b])= \sup_P \sum_{i=1}^{n} |f(t_i)-f(t_{i-1})|,$$ where the supremum is taken over all partitions $P$ of the interval $[a,b].$\\ If $V(f,[a,b]) < \infty,$ we say that $f$ is of bounded variation. The set of all functions of bounded variation on $[a,b]$ will be denoted by $\mathcal{BV}([a,b])$. Note that the space $\mathcal{BV}([a,b])$ is a Banach space when equipped with the norm $\|f\|_{\mathcal{BV}}:= |f(a)| +V(f,[a,b]).$
\end{definition}
In the sequel, we use the following theorem to prove bounded variation, see, for instance, \cite{Gordon}.
\begin{theorem}\label{ET2}
A function $f$ is of bounded variation on an interval $[a,b]$ if and only if it can be decomposed as a difference of two non-decreasing functions.
\end{theorem}
\begin{definition}\label{def1}
Let $f$ be a continuous function on $[a,b]$ and $x_0 \in (a,b)$. The point $x_0$ is said to be a bounded variation point of $f$, if there exists an interval $J=[c,d]$, $a \le c<x_0< d \le b$, such that $f$ is of bounded variation on $J$. Otherwise, $x_0$ is referred to as a point of unbounded variation (obvious modification if $x_0$ is an end point of $[a,b]$).
\end{definition}
\begin{remark}
If $f$ has no unbounded variation points
on $[a, b]$, then $f$ is of bounded variation on $[a, b]$. On
the other hand, if $f$ has at least one unbounded
variation point on $[a, b]$, then $f$ must be of unbounded
variation on $[a, b]$.
\end{remark}
\begin{definition}
Let $f: [a,b] \to \mathbb{R}$ be an integrable function on $[a,b].$ The Riemann-Liouville fractional integral of $f$ of order $\alpha>0$ is defined as $$  \mathfrak{I}_{a}^{\alpha}f(x)=\frac{1}{\Gamma (\alpha)} \int_a ^x (x-t)^{\alpha-1} f(t)~dt.$$
\end{definition}
The reader is encouraged to consult the book \cite{SKM} and the research article \cite{P2} for details of various fractional integrals and their possible geometrical representations.\\
As a prelude, let us note the following.
\begin{lemma}[\cite{VV1}, Lemma $3.4$] \label{ET3}
If $f:[a,b] \rightarrow \mathbb{R}$ is of bounded variation on $[a,b]$, then the following holds:
\begin{enumerate}
\item  If $f(a) \ge 0,$ then there exist non-decreasing functions $g$ and $h$ such that $f=g - h$, $g(a) \ge 0$ and $h(a)=0.$\\
\item  If $f(a) < 0,$ then there exist non-decreasing functions $g$ and $h$ such that $f=g - h$, $g(a) = 0$ and $h(a) >0.$
\end{enumerate}
\end{lemma}
In the next theorem, we prove that the bounded variation points of a function are preserved by the fractional integral operator.
\begin{theorem} \label{MT1}
Let $0 < a <b < \infty,$ $\alpha >0,$ and $f:[a,b] \to \mathbb{R}$ be a integrable function.
If $x_0 \in [a,b]$ is a bounded variation point of $f,$ then $x_0$ is also a bounded variation point of $\mathfrak{I}_{a}^{\alpha}f.$
\end{theorem}
\begin{proof}
By Definition \ref{def1}, it follows that there exists an interval $J=[a_0,b_0]$ such that $x_0 \in [a_0,b_0],$ and $f$ is of bounded variation on $J.$ Now, our first task is to show that $\mathfrak{I}_{a_0}^{\alpha}f$ can be written as a difference of two non-decreasing functions defined on $J.$  
Since $f$ is of bounded variation on $J,$  there exist non-decreasing functions $g,h:J \to \mathbb{R}$ such that $f(x)=g(x)-h(x)$ for every $x \in J.$ Without loss of generality, we assume $f(a_0) \ge 0.$ With the help of Lemma \ref{ET3}, one can also assume that $g(a_0) \ge 0$ and $h(a_0) = 0.$ From the linearity of fractional integral, it is evident that $\mathfrak{I}_{a_0}^{\alpha}f(x)=G(x)-H(x),$  where $G(x):=~~ \mathfrak{I}_{a_0}^{\alpha}g(x)$ and $H(x):=~~ \mathfrak{I}_{a_0}^{\alpha}h(x).$ We claim that $G$ and $H$ are non-decreasing functions. Now, let $a\le x \le y \le b,$
\begin{equation*}
    \begin{aligned}
     G(y)-G(x)=&\mathfrak{I}_{a_0}^{\alpha}g(y)- \mathfrak{I}_{a_0}^{\alpha}g(x)\\=& \frac{1}{\Gamma (\alpha)} \int_{a_0} ^{y} (y-t)^{\alpha-1} g(t)dt - \frac{1}{\Gamma(\alpha)} \int_{a_0} ^x (x -t)^{\alpha-1} g(t)dt.
      \end{aligned}
               \end{equation*}
Applying a change of variable $x-t=y-u$ in the second integral, we have 
\begin{equation*}
    \begin{aligned}
     G(y)-G(x)=&~ \frac{1}{\Gamma (\alpha)} \int_{a_0}^{y-x+a} (y-t)^{\alpha-1} g(t)dt+~ \frac{1}{\Gamma(\alpha)} \int_{y-x+a_0} ^y (y-t)^{\alpha-1}  \Big[g(t)-
     g(t+x-y)\Big]dt.
      \end{aligned}
               \end{equation*}
 Since $ t+x-y \le t, $ $ g(a_0)\ge 0$ and $g$ is non-decreasing, it is immediate that all terms under the integration signs are non-negative. Therefore, $G(y)-G(x) \ge 0,$ that is, $G$ is a non-decreasing function. Similarly, we establish that $H$ is also a non-decreasing function. If $f(a_0)<0,$ then by Lemma \ref{ET3}, we choose non-decreasing functions $g$, $h:J \to \mathbb{R}$ such that $g(a_0)=0$ and $h(a_0)>0,$  and the rest of the proof for the claim follow on similar lines. Further, without loss of generality, we assume that $f(x) \ge 0 $ for every $x \in [a,a_0].$ Then, using the fact that 
 \[
 \mathfrak{I}_{a}^{\alpha}f(x) = \frac{1}{\Gamma (\alpha)} \int_{a}^{a_0} (x-t)^{\alpha-1} f(t)dt+ \mathfrak{I}_{a_0}^{\alpha}f(x) = G(x) - \Big(H(x) - \frac{1}{\Gamma (\alpha)} \int_{a}^{a_0} (x-t)^{\alpha-1} f(t)dt\Big),
 \]
 for $a_0 \le x \le y$, we have $G(y) -G(x) \ge 0$ and
 \begin{equation*}
 \begin{aligned}
  & \Big(H(y)  - \frac{1}{\Gamma (\alpha)} \int_{a}^{a_0} (y-t)^{\alpha-1} f(t)dt\Big)- \Big(H(x) - \frac{1}{\Gamma (\alpha)} \int_{a}^{a_0} (x-t)^{\alpha-1} f(t)dt\Big) \\ &= H(y)- H(x)+ \frac{1}{\Gamma (\alpha)} \int_{a}^{a_0} \Big((x-t)^{\alpha-1} - (y-t)^{\alpha-1}\Big)f(t)dt\ge 0.
  \end{aligned}
 \end{equation*}
 That is, $\mathfrak{I}_{a}^{\alpha}f$ can be written as a difference of two non-decreasing functions on $J.$ 
 Thus, the proof of the theorem is complete.
 \end{proof}
 Further, we introduce an interesting result as follows: 
 \begin{theorem}
 The fractional operator $\mathfrak{I}_{a}^{\alpha}: \mathcal{BV}([a,b]) \cap \mathcal{C}([a,b]) \to \mathcal{BV}([a,b])$ is a bounded linear operator.
 \end{theorem}
 \begin{proof}
 By Theorem \ref{MT1}, the operator $\mathfrak{I}_{a}^{\alpha}$ is well-defined. Obvious also that $\mathfrak{I}_{a}^{\alpha}$ is a linear operator. Now, it remains to show that the operator is bounded. For this,
 let $P: a=x_0<x_1<x_2 < \dots <x_N =b$ be a partition of $[a,b].$ Then, we have
 \begin{equation*}
     \begin{aligned}
      \mathfrak{I}_{a}^{\alpha}f(x_i)- \mathfrak{I}_{a}^{\alpha}f(x_{i-1})= &~ \frac{1}{\Gamma (\alpha)} \int_{a}^{x_i-x_{i-1}+a} (x_i-t)^{\alpha-1} f(t)dt \\ & + ~ \frac{1}{\Gamma(\alpha)} \int_{x_i-x_{i-1}+a} ^{x_i} (x_i-t)^{\alpha-1}  \Big[f(t)-
      f(t+x_{i-1}-x_i)\Big]dt.
     \end{aligned}
 \end{equation*}
 Consequently, 
 \begin{equation*}
      \begin{aligned}
     \sum_{i=1}^{N} \Big| \mathfrak{I}_{a}^{\alpha}f(x_i)- \mathfrak{I}_{a}^{\alpha}f(x_{i-1})\Big| \le  &~ \frac{\|f\|_{\infty}}{\Gamma (\alpha)} \sum_{i=1}^{N} \Big| \int_{a}^{x_i-x_{i-1}+a} (x_i-t)^{\alpha-1}dt \Big| \\ & + ~ \sum_{i=1}^{N} \frac{\big|f(t_i)-
            f(t_i+x_{i-1}-x_i)\big|}{\Gamma(\alpha)} \Big| \int_{x_i-x_{i-1}+a} ^{x_i} (x_i-t)^{\alpha-1}  dt \Big|\\
      \le  &~ \frac{\|f\|_{\infty}}{\Gamma (\alpha+1)}\sum_{i=1}^{N} \Big[  (x_i-a)^{\alpha} - (x_{i-1}-a)^{\alpha} \Big] \\ & + ~ \sum_{i=1}^{N} \frac{\big|f(t_i)-
                  f(t_i+x_{i-1}-x_i)\big|}{\Gamma(\alpha+1)}  (x_{i-1}-a)^{\alpha} \\ \le  &~ \frac{\|f\|_{\infty}}{\Gamma (\alpha+1)}\sum_{i=1}^{N} \Big[  (x_i-a)^{\alpha} - (x_{i-1}-a)^{\alpha} \Big] \\ & + ~ \frac{(b-a)^{\alpha}}{\Gamma(\alpha+1)} \sum_{i=1}^{N} \big|f(t_i)-
                     f(t_i+x_{i-1}-x_i)\big|\\ \le  &~ \frac{\|f\|_{\infty}}{\Gamma (\alpha+1)} V(g,[a,b])  + ~ \frac{(b-a)^{\alpha}}{\Gamma(\alpha+1)} V(f,[a,b])   
                     \\ \le  &~ 2 \max \Big\{\frac{V(g,[a,b])}{\Gamma (\alpha+1)} ,\frac{(b-a)^{\alpha}}{\Gamma(\alpha+1)}\Big\} ~~ \|f\|_{\mathcal{BV}},                
      \end{aligned}
  \end{equation*}
  in the above $g(x)=(x-a)^{\alpha},$ $\|f\|_{\infty}\le \|f\|_{\mathcal{BV}},$ and 
  \[
  \big|f(t_i)- f(t_i+x_{i-1}-x_i)\big|=\sup_{  t \in [x_i -x_{i-1}+a,~ x_i]} \big|f(t)- f(t+x_{i-1}-x_i)\big|,
  \]
  are used. Note that existence of $t_i$ is immediate from the compactness of interval $[x_i -x_{i-1}+a,x_i]$ and continuity of the function $f.$
   Since the previous expression holds for arbitrary partition, we have 
  \[
  \| \mathfrak{I}_{a}^{\alpha}f \|_{\mathcal{BV}} \le 2 \max \Big\{\frac{V(g,[a,b])}{\Gamma (\alpha+1)} ,\frac{(b-a)^{\alpha}}{\Gamma(\alpha+1)}\Big\} ~~ \|f\|_{\mathcal{BV}},
  \]
  establishing the proof.
 \end{proof}
 Let us write the semigroup property of the fractional integral for an integrable function $f:[a,b] \to \mathbb{R}$, see \cite{SKM} for more details, namely,
  $$\mathfrak{I}_{a}^{\alpha} ~~ ~~\mathfrak{I}_{a}^{\beta}f =~~~ \mathfrak{I}_{a}^{\alpha+\beta}f.$$
  	Recall that for a given integrable function $f:[a,b] \to \mathbb{R}$, the function $F:[a,b] \to 
  \mathbb{R}$ defined by $F(x)= \int_{a}^{x} f(t) dt,$ is of bounded variation on $[a,b].$ That is, $F$ has no unbounded variation points. \\ 
  Now, we are ready to prove our main result.
 \begin{theorem}
 Let $f \in \mathcal{C}([a,b]),$ and $ 0 < a < b < \infty .$
 \begin{enumerate}
 \item  If $ 0< \alpha < 1,$ then $\mathfrak{I}_{a}^{\alpha}f$ has the number of unbounded variation points at most equal to the number of unbounded variation points of $f.$  
  \item If $  \alpha \ge 1,$ then there is no unbounded variation points of $\mathfrak{I}_{a}^{\alpha}f.$
 \end{enumerate}
 \end{theorem}
 \begin{proof}
 Let $ 0< \alpha < 1$ and $x_0 \in [a,b]$. In view of Theorem \ref{MT1}, the point $x_0$ will be a bounded variation point of $\mathfrak{I}_{a}^{\alpha}f$ whenever it is so of $f.$ Hence, the first part of the theorem is established. Further, the semigroup property yields the second part.  
 \end{proof}
 Let us now give an example of a function which has one unbounded variation point.
 \begin{example}
 Let $f:[0,1] \to \mathbb{R}$ be a function defined by 
 \begin{equation*}
  f(x)= \begin{cases}
    \sin(\frac{1}{x}), \quad~\text{if}~~0<x\le 1\\
       0, ~\quad \text{if}~~~~x = 0.
  \end{cases}
  \end{equation*}
  Note that $f$ is integrable and not of bounded variation on $[0,1].$ To be precise, $0$ is an unbounded variation point of $f.$ Now, for $\alpha \ge 1,$  $\mathfrak{I}_{0}^{\alpha}f$ has no unbounded variation points.
 \end{example}
\subsection*{Acknowledgements}
   The first author expresses his gratitude to the University Grants
   Commission (UGC), India, for financial support.
\bibliographystyle{amsplain}

\begin{thebibliography}{10}


\bibitem{B} M. F. Barnsley, Fractals Everywhere, 2nd edition, Academic Press, Boston, 1993.
\bibitem {Gordon}  R. A. Gordon, Real Analysis: A First Course, 2nd edition, Boston, Pearson Education Inc., 2002.


\bibitem {Liang11} Y. S. Liang, Fractal dimension of Riemann-Liouville fractional integral of $1$-dimensional continuous functions, Fractional Calculus and Applied Analysis, 21(6) (2019) 1651-1658.

    \bibitem {Liang2} Y. S. Liang, Definition and classification of one-dimensional continuous functions with unbounded variation, Fractals 25 (2017), 12 pages.
    
            \bibitem {Liang4} Y. S. Liang, Some remarks on continuous functions of unbounded variation, Acta Math. Sin. 59 (2016) 215-232.
    
\bibitem {Liang1} Y. S. Liang, Box dimensions of Riemann-Liouville fractional integrals of continuous functions of bounded variation, Nonlin. Anal. 72 (2010) 4304-4306.


    \bibitem {LX} Y. Li and W. Xiao, Fractal dimension of Riemann-Liouville fractional integral of certain unbounded variational continuous function, Fractals 25 (2017), 7 pages.
    \bibitem {PM1} P. R. Massopust, Fractal Functions, Fractal Surfaces, and Wavelets, 2nd ed., Academic Press, San Diego, 2016.
 \bibitem{P2} I. Podlubny, Geometric and physical
interpretation of fractional integration
and fractional differentiation, Fract.
Calc. Appl. Anal. 5 (2002) 367-386.



\bibitem{SKM} S. G. Samko, A. A. Kilbas, O. I. Marichev, Fractional Integrals and Derivatives, Theory and Applications, Gordon and Breach, Yverdon et alibi, 1993.
\bibitem {WY} J. Wang, K. Yao, Dimension analysis of continuous functions of unbounded variation, Fractals, 25 (2017), 6 pages.
\bibitem{VV2} S. Verma, P. Viswanathan, Bivariate functions of bounded variation: Fractal dimension and fractional integral,
Indagationes Mathematicae 31 (2020) 294-309.

\bibitem{VV1} S. Verma, P. Viswanathan, A note on Katugampola fractional calculus and fractal dimensions, Applied Mathematics
and Computation 339 (2018), 220-230.
\bibitem {WD} X.E. Wu, J.H. Du, Box dimension of Hadamard fractional integral of continuous functions of bounded and unbounded variation, Fractals, 25 (2017), 7 pages.
\bibitem {Z} Q. Zhang, Some remarks on one-dimensional functions and their Riemann-Liouville fractional calculus, Acta Math. Sin. 30 (2014) 517-524.


\end{thebibliography}

\end{document}